\documentclass[a4paper,12pt]{amsart}

\usepackage{amssymb}
\usepackage{amsmath}
\usepackage{amsthm}
\usepackage{mathtools}
\usepackage{enumitem}
\usepackage{graphicx}
\usepackage{tikz}
\usepackage{mathrsfs}
\usepackage[mathscr]{euscript}
\usepackage{tkz-graph}

\usepackage{caption}
\usepackage[skip=1ex, belowskip=2ex]{subcaption}
\usepackage[export]{adjustbox}

\usetikzlibrary{arrows}
\usetikzlibrary{tikzmark}
\usetikzlibrary{cd}
\usetikzlibrary{calc}
\graphicspath{ {.} }

\usepackage[english]{babel}

\newcommand{\innerthmname}{}

\theoremstyle{definition}

\usetikzlibrary{positioning}
\newtheorem{theorem}[equation]{Theorem}
\newtheorem{lemma}[equation]{Lemma}

\newtheorem{corollary}[equation]{Corollary}

\theoremstyle{definition}
\newtheorem{definition}[equation]{Definition}

\theoremstyle{remark}

\newtheorem{remark}[equation]{Remark}

\numberwithin{equation}{section}
\usepackage{fancyhdr}
\usepackage{amsfonts}
\usepackage{hyperref}
\usepackage{graphicx, xcolor} 
\usepackage{amsthm}

\usepackage[
margin=1in,
marginpar=2cm,
includefoot,
footskip=30pt,
]{geometry}

\makeatletter
\@namedef{subjclassname@2020}{%
	\textup{2020} Mathematics Subject Classification}
\makeatother

\usepackage{scalerel,stackengine}
\stackMath
\newcommand\reallywidehat[1]{%
	\savestack{\tmpbox}{\stretchto{%
			\scaleto{%
				\scalerel*[\widthof{\ensuremath{#1}}]{\kern-.6pt\bigwedge\kern-.6pt}%
				{\rule[-\textheight/2]{1ex}{\textheight}}
			}{\textheight}%
		}{0.5ex}}%
	\stackon[1pt]{#1}{\tmpbox}%
}


\title[$R$-Bunce-Deddens Algebras]{$R$-Bunce-Deddens Algebras}
\author[A. Zhang]{Allen Zhang}
\email{allenusca@gmail.com}

\begin{document}
	\begin{abstract}
        We construct an $R$-algebraic analog of the Bunce-Deddens algebra using the theory of Leavitt labelled path algebras. In particular, we will construct a labelled space whose associated partial action on tight filters is exactly the odometer action on the Cantor set that induces the classical Bunce-Deddens algebra. After defining $R$-Bunce-Deddens algebras, we will prove $R$-algebraic analogs of various results for classical Bunce-Deddens algebras. As our primary application of these results, we will show that, for any field $K$, a $K$-Bunce-Deddens algebra is not Morita equivalent to any Leavitt path algebra.
	\end{abstract}
	\maketitle
	
	\section{Introduction}
    Bunce-Deddens algebras were first studied by Bunce and Deddens in \cite{bunce1975family} and have since been extensively researched \cite{hawkins2013spectral,orfanos2010generalized,rout2017classification}. Particularly relevant to this paper is a well-known result \cite{davidson1996c, exel1994bunce} that constructs the Bunce-Deddens algebra as a crossed product arising from an odometer action on a Cantor set.

    In the study of $C^{\ast}$-algebras, there has been considerable interest in properties of specific classes of $C^{\ast}$-algebras. In particular, graph $C^{\ast}$-algebras have received a large amount of attention since their introduction \cite{cuntz1980class,kumjian1998cuntz}. Because of the success of graph $C^{\ast}$-algebras, Bates and Pask introduced the more general combinatorial class of labelled spaces and an associated class of labelled space $C^{\ast}$-algebras \cite{bates2007c, bates2009c}.

    During the study of combinatorial $C^{\ast}$-algebras, it was discovered that an algebraic analog preserved many interesting properties of these algebras. From this, the Leavitt path algebra was defined independently in \cite{abrams2005leavitt} and \cite{ara2004fractional} and has since grown into an important field of algebra in its own right. Eventually, an $R$-algebraic analog of the labelled space $C^{\ast}$-algebra was defined in \cite{Boava2021LeavittPA} and termed a Leavitt labelled path algebra.

	In this paper, we will define an $R$-algebraic analog of the Bunce-Deddens algebra $B_R(\{n_k\})$ from a Leavitt labelled path algebra. We will use the theory developed for Leavitt labelled path algebras to show that this is a natural analog to the $C^{\ast}$-algebraic case by proving that $B_R(\{n_k\})$ can be realized as the algebraic analog of the crossed product that defines the Bunce-Deddens $C^{\ast}$-algebra. 
    
    After defining these algebras, we will prove the following results for $R$-Bunce-Deddens algebras:
    \begin{enumerate}
    \item (Theorem~\ref{simplicitytheorem}) When $R$ is a field, $B_R(\{n_k\})$ is simple.
    \item (Theorem~\ref{matrixalgebrarep}) $B_R(\{n_k\})$ is the direct limit of matrix rings over the Laurent polynomial ring.
    \item (Corollary~\ref{notlocallymatricialtheorem}) When $R$ is a field, $B_R(\{n_k\})$ is not locally matricial (with respect to $R$).
    \item (Theorem~\ref{stablyfinitethm}) $B_R(\{n_k\})$ is stably finite.
    \item (Theorem~\ref{notvonneumannregular}) $B_R(\{n_k\})$ is not von Neumann regular.
    \end{enumerate}
    These are all analogs of results known in the $C^{\ast}$-algebraic case.

    As our main application, we will prove that for any unital commutative ring $R$, directed graph $G$, and field $K$, there exists no Leavitt path algebra $L_R(G)$ that is Morita equivalent to $B_K(\{n_k\})$. In doing so, we show that the class of Leavitt labelled path algebras is strictly larger than the class of Leavitt path algebras, even up to Morita equivalence. This is the algebraic equivalent of a result known for graph and labelled space $C^{\ast}$-algebras \cite{jeong2017finite, kim2016unique}.
	
	\section{Constructing the $R$-Algebra}

	We will briefly review the definition of weakly left-resolving normal labelled spaces and their associated Leavitt labelled path algebras from \cite{Boava2021LeavittPA}. Throughout, let $R$ be a unital commutative ring.
	
    A labelled graph is a directed graph $\mathcal E = (\mathcal E^0, \mathcal E^1)$ paired with a labelling $\mathcal L: \mathcal E^1 \rightarrow \mathcal A$ where $\mathcal A$ is an alphabet. Labelling extends to finite-length paths in $\mathcal E^{\ast}$.
	
	We denote $\mathcal L^{\ast} = \mathcal L(\mathcal E^{\ast})$. For any subset $A \subseteq \mathcal E^0$ and $\alpha \in \mathcal L^{\ast}$, we define $r(A, \alpha)$ to be the range of paths $p = e_1\ldots e_{|\alpha|}$ that start in $A$. More precisely, $r(A, \alpha) = \{r(p) \colon p \in \mathcal E^{\ast}, \mathcal L(p) = \alpha, s(e_1) \in A\}$. We define $r(\alpha) = r(\mathcal E^0, \alpha)$.
	
	Let $\mathcal B$ be a set of subsets of $\mathcal E^0$ closed under finite unions and intersections that contains $r(\alpha)$ for all $\alpha \in \mathcal L^{\ast}$ and such that for all $A \in \mathcal B$ and $\alpha \in \mathcal L^{\ast}$ we have that $r(A, \alpha) \in \mathcal B$.
	
	We call a triple $(\mathcal E, \mathcal L, \mathcal B)$ a \textit{labelled space}.
	
	A labelled space is \textit{weakly left-resolving} if for all $A, B \in \mathcal B$ and $\alpha \in \mathcal L^{\ast}$, we have that $r(A \cap B, \alpha) = r(A, \alpha) \cap r(B, \alpha)$. A labelled space is \textit{normal} if $\mathcal B$ is closed under relative complements. We assume that all labelled spaces considered henceforth are weakly left-resolving and normal and just refer to them as labelled spaces.
	
	For $\alpha \in \mathcal L^{\ast}$, we define $\mathcal B_{\alpha} = \mathcal B \cap \mathcal P(r(\alpha))$. Define $\mathcal B_{\omega} = \mathcal B$. $\mathcal B_{\alpha}$ is an ideal of $\mathcal B$.
	
	For $A \in \mathcal B$, define $\Delta_A = \{a \in \mathcal A \colon r(A, a) \neq \emptyset\}$. Let $\mathcal E_{\text{sink}}$ be the sinks of $\mathcal E$ (as a directed graph). Note that $\mathcal E_{\text{sink}}$ is not necessarily in $\mathcal B$.
	
	For $A \in \mathcal B$, call $A$ \textit{regular} if $0 < |\Delta_A| < \infty$ and there exists no $\emptyset \neq B \in \mathcal B$ such that $B \subseteq A \cap \mathcal E_{\text{sink}}$. If $A$ is not regular, we refer to it as \textit{singular}. We denote $\mathcal B_{\text{reg}} \subseteq \mathcal B$ to be the regular sets. $\mathcal B_{\text{reg}}$ is an ideal of $\mathcal B$.
	
	The associated Leavitt labelled path algebra $L_R(\mathcal E, \mathcal L, \mathcal B)$ is defined as the associative $R$-algebra on the generators $\{p_A\}_{A \in \mathcal B} \cup \{s_a, s_a^{\ast}\}_{a \in \mathcal A}$ that satisfy the following relations, which are often known as the Cuntz-Krieger relations.
	
	\begin{enumerate}
		\item $p_{A \cap B} = p_A p_B$, $p_{A \cup B} = p_{A} + p_B - p_{A \cap B}$, $p_{\emptyset} = 0$ for $A, B \in \mathcal B$
		\item $p_A s_a = s_a p_{r(A, a)}$ and $s_a^{\ast}p_A = p_{r(A, a)}s_a^{\ast}$ for $A \in \mathcal B$ and $a \in \mathcal A$
		\item $s_a^{\ast}s_b = \delta_{a, b} p_{r(a)}$ for $a, b \in \mathcal A$
		\item $s_a s_a^{\ast} s_a = s_a$ and $s_a^{\ast}s_as_a^{\ast} = s_a^{\ast}$ for $a \in \mathcal A$.
		\item For all $A \in \mathcal B_{\text{reg}}$, $p_A = \sum_{a \in \Delta_A} s_a p_{r(A, a)}s_a^{\ast}$
	\end{enumerate}

    We now construct our $R$-Bunce-Deddens algebras $B_R(\{n_k\})$ from a Leavitt labelled path algebra.

	\begin{definition}
    Let $\{n_k\}_{k=0}^\infty$ be a strictly increasing sequence of positive integers such that $n_k \mid n_{k+1}$ for $k \geq 0$. Define the vertices $\mathcal E^0 = \mathbb Z$ and let the set of edges be $\{(x, x+1)\}_{x \in \mathbb Z}$ with all edges labelled with a single label $a$, so $\mathcal A = \{a\}$. For $k \geq 0$, define $C_{(k, i)} = i + n_k \mathbb Z$ for $i \in [0, n_k)$. Let $\mathcal B$ be the collection of all finite unions of sets $C_{(k, i)}$ for $k \geq 0$ and $i \in [0, n_k)$. We define \[B_R(\{n_k\}) \coloneqq L_R(\mathcal E, \mathcal L, \mathcal B).\]
	\end{definition} 

	Throughout the rest of this section, fix the sequence $\{n_k\}_{k=0}^{\infty}$. We will prove that $(\mathcal E, \mathcal L, \mathcal B)$ is a weakly left-resolving labelled space. The following lemmas are easy to prove.

	\begin{lemma}\label{repdyadic}
	For every $B \in \mathcal B$, there exists some $k(B)$ such that for all $k \geq k(B)$, we can write $B = \bigsqcup_{i \in S} C_{(k, i)}$ for some $S \subseteq \{0, \ldots, n_k-1\}$.
	\end{lemma}

	\begin{lemma}\label{rangecalc}
	Let $B = \bigsqcup_{i \in S} C_{(k, i)} \in \mathcal B$ for $S \subseteq \{0, \ldots, n_k-1\}$. We have that $r(B, a) = \bigsqcup_{i \in S}C_{(k, i+1)}$ where $i+1$ wraps around to $0$ when $i = n_k-1$.
	\end{lemma}

    \begin{lemma}
    Let $A \in \mathcal B$ and $m \geq 0$. There exists exactly one $A' \in \mathcal B$ such that $r(A', a^m) = A$. Motivated by this, we define $r(A, a^{-m}) = A'$.
    \end{lemma}

	\begin{theorem}
	$(\mathcal E, \mathcal L, \mathcal B)$ is a weakly left-resolving normal labelled space.
	\end{theorem}
	\begin{proof}
    Using Lemma~\ref{repdyadic}, for two sets $A, B \in \mathcal B$, we choose $k' = \max(k(A), k(B))$. Because all sets $C_{(k', i)}$ for $i \in [0, n_{k'})$ are mutually disjoint, the union, intersection, and relative complement can be clearly constructed from the representation in Lemma~\ref{repdyadic}, so the space is normal. 
    
    To prove that $\mathcal B$ is closed under relative ranges, again using the representation $B = \bigsqcup_{i \in S} C_{(k, i)}$, it is clear that $r(B, a) = \bigsqcup_{i \in S} C_{(k, i+1)} \in \mathcal B$ using Lemma~\ref{rangecalc}, reducing modulo $n_k$ if $i+1 = n_k$. Hence, $r(B, a) \in \mathcal B$ and because $\mathcal A = \{a\}$, $\mathcal B$ is therefore closed under relative ranges.
    
    To see that the space is weakly left-resolving, again, taking $k' = \max(k(A), k(B))$ so $A = \bigsqcup_{i \in S_A} C_{(k', i)}$ and $B = \bigsqcup_{i \in S_B} C_{(k', i)}$, we can calculate that $A \cap B = \bigsqcup_{i \in S_A \cap S_B} C_{(k', i)}$. Using the same calculation as for the relative ranges, we find that $r(A \cap B, a) = r(A, a) \cap r(B, a)$, so the space is weakly left-resolving.
	\end{proof}

    Hence, $(\mathcal E, \mathcal L, \mathcal B)$ is a weakly left-resolving normal labelled space, so we can define \[B_R(\{n_k\}) \coloneqq L_R(\mathcal E, \mathcal L, \mathcal B).\]

    \section{Connection to the classical Bunce-Deddens Algebra}
    We first review the construction of the classical Bunce-Deddens algebra, which can be found in \cite{davidson1996c}. Throughout this section, let $\{n_k\}_{k=0}^{\infty}$ be a strictly increasing sequence of positive integers where $n_k \mid n_{k+1}$ as before. 

    \begin{definition}\label{odometerdefinition}
    Consider the topological space \[X = \{(x_0, x_1, x_2, \ldots) \in \prod_{k=0}^\infty \{0, \ldots, n_k-1\} \colon x_k \equiv x_{k+1} \mod n_k\} \] equipped with the subspace topology from the discrete product topology. Define an automorphism $\varphi: X \rightarrow X$ that maps \[(x_0, x_1, \ldots) \mapsto (x_0+1, x_1+1, \ldots)\] with the operation being taken modulo $n_k$ at every level $k$. The Cantor space $X$ is compact Hausdorff and the map $\varphi$ induces a topological action $\mathbb Z \curvearrowright X$ in the standard way.
    \end{definition}

    \begin{definition}
    Let $C(X) = \{f: X \rightarrow \mathbb C \colon f \textit{ is continuous}\}$. The Bunce-Deddens $C^{\ast}$-algebra $B(\{n_k\})$ is the crossed product $C(X) \rtimes_{\tilde{\varphi}} \mathbb Z$ where $\tilde{\varphi}$ is the induced action on $C(X)$ by taking $f \mapsto f \circ \varphi^{-1}$.
    \end{definition}

    We now review some theory required to derive this crossed product from our labelled space $(\mathcal E, \mathcal L, \mathcal B)$. Our main references for these results will be \cite{boava2017inverse, de2020labelled}.

  	\begin{definition}
    Let $P$ be a meet semilattice with zero. For a subset $F \subseteq P$, we say that
    \begin{enumerate}
    \item $F$ is a \textit{filter} if it is neither empty nor $P$ and, for all $x, y \in F$, we have that $x \land y \in F$ and, for all $y \in P$ such that $x \leq y$ for some $x \in F$, we have that $y \in F$.
    \item $F$ is an \textit{ultrafilter} if it is a filter that is maximal under inclusion.
    \end{enumerate}
    On the set of filters $F(P)$, we define a topology generated by sets of the form $U_{(x \colon x_1, \ldots, x_n)} = \{F \in F(P) \colon x \in F, x_1 \notin F, \ldots, x_n \notin F\}$ for any $x, x_1, \ldots, x_n \in P$ and $n \geq 0$. These sets $U_{(x \colon x_1, \ldots, x_n)}$ form a basis of compact open sets and the topology on $F(P)$ is Hausdorff.

    We define the \textit{tight filters} $T(P)$ as the closure of the ultrafilters in $F(P)$ with the subspace topology inherited from the topology on $F(P)$. We define $V_{(x \colon x_1, \ldots, x_n)}  = U_{(x \colon x_1, \ldots, x_n)} \cap T(P)$. The sets \[\{V_{(x \colon x_1, \ldots, x_n)}\colon n \geq 0 \text{ with } x, x_1, \ldots, x_n \in P\}\] form a basis of compact open sets of $T(P)$.
	\end{definition}

    The next lemma will allow us to determine whether a filter is an ultrafilter (or a tight filter). 

    \begin{lemma}\cite[Theorem~3.8]{zhang2026partial} \label{computableconditionlemma}
    For any $x \in P$, a \textit{finite cover} of $x$ is a set $\{x_1, \ldots, x_n\}$ such that $x_i \leq x$ and for all $0 \neq y \leq x$, there exists some $i$ such that $x_i \land y \neq 0$.

    A filter $\xi$ is an ultrafilter if and only if for all $x \in P \setminus \xi$, there exists some $y \in \xi$ such that $y \land x = 0$.

    A filter $\xi$ is a tight filter if and only if for all $x \in \xi$ and finite covers $\{x_1, \ldots, x_n\}$ of $x$, there exists some $i$ such that $x_i \in \xi$.
    \end{lemma}

    \begin{theorem}
    For every labelled space $(\mathcal E, \mathcal L, \mathcal B)$ there is an associated inverse semigroup \[S = \{(\alpha, A, \beta) \colon \alpha, \beta \in \mathcal L^{\ast} \text { and } A \in \mathcal B_{\alpha} \cap \mathcal B_{\beta} \text { and } A \neq \emptyset\} \cup \{0\}\] with idempotents 
    \[E(S) = \{(\alpha, A, \alpha) \colon \alpha \in \mathcal L^{\ast} \text { and } A \in \mathcal B_{\alpha} \text { and } A \neq \emptyset\} \cup \{0\}.\]

    $E = E(S)$ is a meet-semilattice and its tight filters $T(E)$ can be characterized by \cite[Theorem~6.7]{boava2017inverse}. Let $\mathbb F$ be the free group generated by the alphabet $\mathcal A$ associated to the labelled space. By \cite[Proposition~3.12]{de2020labelled}, there is a partial action\[\Phi = (\{V_t\}_{t \in \mathbb F}, \{\phi_t\}_{t \in \mathbb F})\] defined by \[V_{\alpha} = V_{(\alpha, r(\alpha), \alpha)}\]\[V_{\alpha^{-1}} = V_{(\omega, r(\alpha), \omega)}\]\[V_{\omega} = T(E)\] for $\alpha \in \mathcal L^{\geq 1}$ and the maps $\phi_t: \mathbb V_{t^{-1}} \rightarrow V_t$ are defined by gluing actions on tight filters (see \cite[Section~2.5]{boava2017inverse}).

    \end{theorem}
    We now describe the semigroup induced by the labelled space associated to $\{n_k\}_{k=0}^{\infty}$. Fix such a sequence and let $(\mathcal E, \mathcal L, \mathcal B)$ be the induced labelled space, $S$ be the induced semigroup, and $E(S)$ be the meet semilattice of idempotents. Let $X$ be the Cantor space induced by the sequence $\{n_k\}_{k=0}^{\infty}$.

    Note that $\mathcal B_{\alpha} = \mathcal B$ for all $\alpha \in \mathcal L^{\ast}$. Hence, \[E(S) = \{(\alpha, A, \alpha) \colon \alpha \in \mathcal L^{\ast} \text{ and } A \in \mathcal B \setminus \{\emptyset\}\} \cup \{0\}.\] Because $\mathcal A$ has only one symbol, we take $m = |\alpha|$ and write $(m, A, m)$ for elements in $E(S)$. With this notation and by \cite[Proposition~4.1]{de2020labelled}, we have that $(m_1, A_1, m_1) \leq (m_2, A_2, m_2)$ if and only if $m_1 \geq m_2$ and $A_1 \subseteq r(A_2, a^{m_1-m_2})$.

    \begin{theorem}
    There is a topological isomorphism $f: T(E) \rightarrow X$. The map $f$ takes
    \[\xi \mapsto (\xi_0, \xi_1, \xi_2, \ldots)\]  where $\xi_k = \{i \in [0, n_k) \colon (0, C_{(k, i)}, 0) \in \xi\}$. We will prove that if $\xi$ is a tight filter, then $|\xi_k| = 1$ and use $\xi_k$ to denote this unique element.
    \end{theorem}
    \begin{proof}
    If $|\xi_k| > 1$, taking distinct $i, i' \in \xi_k$, we have that $(0, C_{(k, i)}, 0)(0, C_{(k, i')}, 0) = (0, \emptyset, 0) = 0 \in \xi$. Hence, because a filter cannot contain $0$, we must have that $|\xi_k| \leq 1$.

    We now show that $|\xi_k| \neq 0$. We know that $(0, \mathbb Z, 0) \in \xi$ because it is the highest element. Using Lemma~\ref{computableconditionlemma} and noting that $\{(0, C_{(k, i)}, 0)\}_{i=0}^{n_k-1}$ is a finite cover of $(0, \mathbb Z, 0)$, we find that at least one of $\{(0, C_{(k, i)}, 0)\}_{i=0}^{n_k-1}$ must be in $\xi$ so $|\xi_k| \geq 1$.

    To see that $(\xi_k)_{k=0}^{\infty}$ satisfies $\xi_{k+1} \equiv \xi_k \mod n_k$, note that if this is not true then $C_{(k+1, \xi_{k+1})} \cap C_{(k, \xi_k)} = \emptyset$ which implies that $(0, C_{(k+1, \xi_{k+1})} , 0)(0,  C_{(k, \xi_k)}, 0) = 0 \in \xi$ as before, which is a contradiction.

    To prove that $f$ is bijective, we construct an inverse. For a given sequence $(x_0, x_1, \ldots)$ such that $x_{k+1} \equiv x_k \mod n_k$, consider the set $\{(m, r(C_{(k, x_k)}, a^m), m)\}_{k, m \geq 0}$. Any finite product of these elements is non-zero, so the set of all elements in $E(S)$ larger than some element in the set, denoted by $\{(m, r(C_{(k, x_k)}, a^m), m)\}_{k, m \geq 0}^{+}$, is a well-defined filter. This process is clearly an inverse of $f$, so if $\xi = \{(m, r(C_{(k, x_k)}, a^m), m)\}_{k, m \geq 0}^{+}$ is a tight filter then we have proved a bijection.

    To show that $\xi$ is a tight filter, we will show that it is an ultrafilter. To do this, we will use Lemma~\ref{computableconditionlemma}. Let $0 \neq (m, A, m) \in E(S) \setminus \xi$ and write $A = \bigsqcup_{i \in S} C_{(k, i)}$ for sufficiently large $k$ and some set $S \subseteq \{0, 1, \ldots, n_k-1\}$ by Lemma~\ref{repdyadic}. 
    
    Consider the element $(n, r(A, a^{n-m}), n)$ for some $n$ where $n \geq m$ and $n_k \mid n$. Because $n$ is divisible by $n_k$, we have that $r(C_{(k, x_k)}, a^n) = C_{(k, x_k)}$ so $(n, C_{(k, x_k)}, n) \in \xi$. Note that by the representation of $A$, we either have that $C_{(k, x_k)} \subseteq r(A, a^{n-m})$ or $r(A, a^{n-m}) \cap C_{(k, x_k)} = \emptyset$. If $C_{(k, x_k)} \subseteq r(A, a^{n-m})$, then we would have that $(n, r(A, a^{n-m}), n) \geq (n, C_{(k, x_k)}, n)$ which would imply that $(n, r(A, a^{n-m}), n) \in \xi$. 
    
    Thus, we find that \[(0, C_{(k, x_k)}, 0)(n, r(A, a^{n-m}), n) = (n, C_{(k, x_k)} \cap r(A, a^{n-m}), n) = 0.\] Now note that $(n, \mathbb Z, n) \in \xi$ and \[(n, \mathbb Z, n)(m, A, m) = (n, r(A, a^{n-m}), n)\] so we find that \[0 = (0, C_{(k, x_k)}, 0)(n, r(A, a^{n-m}), n) = (0, C_{(k, x_k)}, 0)(n, \mathbb Z, n)(m, A, m)\] with $ (0, C_{(k, x_k)}, 0)(n, \mathbb Z, n) \in \xi$, so our condition is met and $\xi$ is an ultrafilter.

    Both $X$ and $T(E)$ are Hausdorff. $X$ is compact by Tychonoff's theorem and $T(E)$ is compact because $T(E) = V_{(0, \mathbb Z, 0)}$ is a compact open set. Thus, to prove that $f$ is a topological isomorphism it suffices to prove that it is continuous. For cylinders on indices $k_1, \ldots, k_m$ with elements $x_{k_j} \in [0, n_{k_j})$, the inverse image can be calculated to be exactly $\bigcap_{j=1}^m V_{(0, C_{(k_j, x_{k_j})}, 0)}$ which is clearly an open set, proving continuity.
    \end{proof}
    
    \begin{corollary}\label{filterrepresentation}
    Every tight filter $\xi$ is uniquely determined by the set $A(\xi) = \{A \colon (0, A, 0) \in \xi\} \subseteq \mathcal B$. Namely, \[\xi = \bigsqcup_{m=0}^{\infty} \{(m, r(A, a^m), m) \colon A \in A(\xi)\}\] and if $(0, A, 0) \in \xi$ then $(m, r(A, a^m), m) \in \xi$.
    \end{corollary}

    \begin{proof}
    Clear from the definition of $f$ and its inverse.
    \end{proof}

    Note that the free group $\mathbb F$ with a single generator is isomorphic to $\mathbb Z$. Furthermore, for all $t \in \mathbb Z$, we have that $V_t = T(E)$ because $(\alpha, r(\alpha), \alpha) = (\alpha, \mathbb Z, \alpha)$ is in all filters, so we have an action $\mathbb Z \curvearrowright T(E)$ instead of a partial action. The following theorem gives an isomorphism between the action on $X$ and the action on $T(E)$.

    \begin{theorem} \label{isomorphicpartialaction}
    Under the isomorphism $f$, we have an isomorphism of actions by $\mathbb Z$ \[\mathbb Z \curvearrowright X \cong \mathbb Z \curvearrowright T(E).\]
    \end{theorem}
    \begin{proof}
    Applying \cite[Definition~5.2]{zhang2026partial}, we can calculate that the map $\phi: T(E) \rightarrow T(E)$ is defined by a partial map $g: E(S) \dashrightarrow E(S)$ that takes $(m, A, m) \mapsto (m-1, A, m-1)$ for $m > 0$ so that $\phi(\xi) = g(\xi)$.

    It suffices to show that the diagram of the generating automorphisms commutes:
    \[\begin{tikzcd}
    T(E) \arrow{r}{f} \arrow[swap]{d}{\phi} & X \arrow{d}{\varphi} \\
    T(E) \arrow[swap]{r}{f}& X
    \end{tikzcd}\]
    Let $\xi \in T(E)$ be a tight filter. We have that $\varphi(f(\xi)) = (\xi_k+1)_{k=0}^{\infty}$ where every coordinate is taken modulo $n_k$. In the other direction, we have that $f(\phi(\xi)) = (g(\xi)_k)_{k=0}^{\infty}$. Note that $(0, C_{(k, \xi_k)}, 0) \in \xi$, so we have that $(1, r(C_{(k, \xi_k)}, a), 1) = (1, C_{(k, \xi_k+1)}, 1) \in \xi$ by Corollary~\ref{filterrepresentation}. Thus, $g((1, C_{(k, \xi_k+1)}, 1)) = (0, C_{(k, \xi_k+1)}, 0)\in g(\xi)$ and $g(\xi)_k = \xi_k+1$ where $\xi_k+1$ is also taken modulo $n_k$. Hence, \[f(\phi(\xi)) = (\xi_k+1)_{k=0}^{\infty} = \varphi(f(\xi))\] so we are done.
    \end{proof}

    \begin{corollary}\cite[cf. Theorem~3.2]{exel1994bunce} Using the isomorphism of the group actions, we obtain the following $R$-algebra isomorphisms.
    \[B_R(\{n_k\}) = L_R(\mathcal E, \mathcal L, \mathcal B) \cong \mathrm{Lc}(X, R) \rtimes_{\tilde{\varphi}} \mathbb Z\cong A_R(X \rtimes_{\varphi} \mathbb Z) \]
        where $A_R(X \rtimes_{\varphi} \mathbb Z)$ is the Steinberg algebra on the groupoid $X \rtimes_{\varphi} \mathbb Z$ and $\mathrm{Lc}(X, R)$ is the $R$-algebra of locally constant functions $X \rightarrow R$ with compact support.
    \end{corollary}
    \begin{proof}
    The result follows from Theorem~\ref{isomorphicpartialaction} together with \cite[Theorem~5.6 and Theorem~6.1]{Boava2021LeavittPA}.
    \end{proof}

    \section{Properties of $B_R(\{n_k\})$}\label{propertiessection}

    We first review some basic definitions for rings with local units. 

	\begin{definition}
	Let $A$ be a (not necessarily unital) ring. We say that $A$ is a ring \textit{with local units} if there is a set of idempotents $E \subseteq A$ such that every finite subset $F \subseteq A$ is contained in $eAe$ for some $e \in E$.
	\end{definition}

    All Leavitt path algebras and Leavitt labelled path algebras are rings with local units, so the definitions apply to all algebras considered in this paper.

    \begin{definition}\cite{Abrams01011983, AnhMoritaEquivalenceLocal}
	For a ring $A$ with local units, the category of unitary left modules consists of those left $A$-modules $M$ that satisfy $AM = M$ where the usual identity condition $1_A \cdot m = m$ is imposed only when $A$ has an identity. For two rings $A, B$ with local units, we say that the rings $A$ and $B$ are Morita equivalent, denoted $A \equiv_M B$, when the categories of unitary left $A$-modules and unitary left $B$-modules are equivalent.
	\end{definition}

    The following definitions will be used later.

	\begin{enumerate}
	\item $A$ is \textit{simple} if its only two-sided ideals are $0$ and $A$.
	\item An element $a \in A$ is \textit{regular} if there exists $b \in A$ such that $aba = a$. $A$ is called \textit{von Neumann regular} if all $a \in A$ are regular.
	\item An idempotent $e \in A$ is called \textit{finite} if every right invertible element of $eAe$ is left invertible. $A$ is called \textit{directly finite} (also commonly called Dedekind finite and von Neumann finite) if every idempotent $e \in A$ is finite. $A$ is called \textit{stably finite} if for every $n \geq 1$, the matrix ring $M_n(A)$ is directly finite.
	\item Let $K$ be a field and let $A$ be a $K$-algebra. $A$ is called \textit{locally matricial} if it is the direct limit of finite direct products of full matrix algebras over $K$.
	\end{enumerate}

    We now prove some properties about $R$-Bunce-Deddens algebras. Throughout, let $\{n_k\}_{k=0}^{\infty}$ be our required family and let $R$ be an arbitrary commutative unital ring. Let $(\mathcal E, \mathcal L, \mathcal B)$ be our associated labelled space.

    \begin{theorem}
    $B_R(\{n_k\})$ is unital.
    \end{theorem}
    \begin{proof}
    The element $\sum_{i=0}^{n_0-1} p_{C_{(0, i)}}$ is clearly a unit because $\mathbb Z = \bigsqcup_{i=0}^{n_0-1} C_{(0, i)}$.
    \end{proof}
    
	\begin{theorem}\cite[cf. Theorem~2]{bunce1975family}\label{simplicitytheorem}
	$B_R(\{n_k\})$ is simple if and only if $R$ is a field.
	\end{theorem}
	\begin{proof}
	We will use the classification of simplicity from \cite[Theorem~9.4]{Boava2021LeavittPA}.

	We first prove that $\{\emptyset\}$ and $\mathcal B$ are the only hereditary saturated subsets. Let $\{\emptyset\} \neq H \subseteq \mathcal B$ be a hereditary saturated subset. It suffices to prove that the vertex set $\mathbb Z \in H$.
	
	Because $H \neq \{\emptyset\}$, there must exist some $C_{(k, i)} \in H$ as all sets of $\mathcal B$ are finite unions of such $C_{(k, i)}$ so we can apply the hereditary condition. Note that, taking $i+j$ modulo $n_k$, we have that $C_{(k, (i+j))} = r(C_{(k, i)}, a^j) \in H$ also by the hereditary property. It is easy to see that $\bigsqcup_{j=0}^{n_k-1} C_{(k, i+j)} = \mathbb Z \in H$, so we are done.

	We now prove that $(\mathcal E, \mathcal L, \mathcal B)$ has no cycles, so the labelled space obeys condition $(L_{\mathcal B})$ trivially. Assume for the sake of contradiction that there exists a path $a^m$ where $m \geq 1$ and a set $A \in \mathcal B$ such that all $B \in \mathcal B$ with $B \subseteq A$ satisfy $r(B, a^m) = B$. Let $C_{(k, i)}$ be any set such that $C_{(k, i)} \subseteq A$. Let $k' \geq 1$ be any integer such that $n_{k'} > m+i$ (using the fact that the sequence is strictly increasing) and $k' \geq k$. We calculate that $r(C_{(k', i)}, a^m) = C_{(k', i+m)}$. Because $n_{k'} > m+i$, we have that $i+m \not\equiv i \mod n_{k'}$ so $C_{(k', i+m)} \neq C_{(k', i)}$ so $r(C_{(k', i)}, a^m) \neq C_{(k', i)}$.  Because $k' \geq k$, we have that $C_{(k', i)} \subseteq C_{(k, i)} \subseteq A$ and thus by the cycle condition we have that $r(C_{(k', i)}, a^m) = C_{(k', i)}$ which is a contradiction.
    \end{proof}

	\begin{theorem}\cite[cf. Proof of Theorem~2]{bunce1975family} \label{matrixalgebrarep}
	Every finite subset $S \subseteq B_R(\{n_k\})$ is contained in a sub-algebra that is isomorphic to $M_{n_k}(R[z, z^{-1}])$ for some $k$. In particular, $B_R(\{n_k\})$ is the direct limit of the following sequence
    
    \[\begin{tikzcd}
    \ldots \arrow{r}{\beta_{k-1}} & M_{n_k}(R[z, z^{-1}]) \arrow{r}{\beta_k} & M_{n_{k+1}}(R[z, z^{-1}]) \arrow{r}{\beta_{k+1}} & \ldots \\
    \end{tikzcd}\] where the maps $\beta_k$ will be defined later. Hence, we can write \[B_R(\{n_k\}) \cong \varinjlim M_{n_k}(R[z, z^{-1}])\] as $R$-algebras.

	\end{theorem}

	\begin{proof}
	Let $k \geq 0$ be fixed and consider the labelled space $(\mathcal E, \mathcal L, \mathcal B_k)$ where $\mathcal B_k \subseteq \mathcal B$ is the Boolean sub-algebra consisting of finite unions of the sets $C_{(k, i)}$ for $i \in [0, n_k)$. It is not hard to show that $(\mathcal E, \mathcal L, \mathcal B_k)$ is another labelled space. In fact, noting that our construction of $(\mathcal E, \mathcal L, \mathcal B)$ does not actually require $\{n_k\}_{k=0}^{\infty}$ to be strictly increasing, $(\mathcal E, \mathcal L, \mathcal B_k)$ is the labelled space that corresponds to the eventually constant sequence $\{n'_j\}_{j=0}^{\infty}$ where $n'_j = n_j$ if $0 \leq j < k$ and $n'_j =  n_k$ if $j \geq k$.

	We now prove that \[L_R(\mathcal E, \mathcal L, \mathcal B_k) \cong M_{n_k}(R[z, z^{-1}]).\] Let $E_{(i, j)}$ for $i, j \in [0, n_k)$ be the standard matrix units whose $(i, j)$-entry is $1$ and whose other entries are $0$. Define a map \[\varphi_k: L_R(\mathcal E, \mathcal L, \mathcal B_k) \rightarrow  M_{n_k}(R[z, z^{-1}])\] by taking \[p_{C_{(k, i)}} \mapsto E_{(i, i)}\] \[s_a \mapsto E_{(0, 1)} + E_{(1, 2)} + \ldots + E_{(n_k-2, n_k-1)} + z E_{(n_k-1, 0)}\] \[s_a^{\ast} \mapsto E_{(1, 0)} + E_{(2, 1)} + \ldots + E_{(n_k-1, n_k-2)} + z^{-1} E_{(0, n_k-1)}.\]
	
	We need to prove that this mapping preserves conditions (1)-(5) from \cite[Definition~3.1]{Boava2021LeavittPA}. Throughout, we will use the fact that $E_{(i, j)} E_{(i', j')} = E_{(i, j')} \cdot 1[j = i']$.
    
    \begin{enumerate}
    \item The diagonal matrix units clearly satisfy the required relations.
    \item It suffices to prove this for all $A = C_{(k, i)}$. Calculate    
    \[p_{C_{(k, i)}} s_a \mapsto \begin{cases} E_{(i, i+1)} & 0 \leq i < n_k-1 \\ z E_{(n_k-1, 0)} & i = n_k-1 \\ \end{cases}\]
    \[s_ap_{C_{(k, i)}} \mapsto \begin{cases} E_{(i-1, i)} & 1 \leq i < n_k \\ z E_{(n_k-1, 0)} & i = 0 \\ \end{cases}\]
    Using \[r(C_{(k, i)}, a) = \begin{cases} C_{(k, i+1)} & 0\leq i < n_k-1 \\ C_{(k, 0)} & i = n_k-1 \end{cases} \] it is obvious that (2) is preserved.
    \item Because $\mathcal A$ has only the symbol $a$, it suffices to prove that $s_a^{\ast}s_a$ maps to the unit matrix as $p_{r(a)}$ is the unit. It is not hard then to see that \[s_a^{\ast}s_a \mapsto \left(E_{(1, 0)} + \ldots + E_{(n_k-1, n_k-2)} + z^{-1} E_{(0, n_k-1)} \right)\left( E_{(0, 1)}  + \ldots + E_{(n_k-2, n_k-1)} + z E_{(n_k-1, 0)} \right) \]\[= \sum_{i=0}^{n_k-1} E_{ii} = I_{n_k}.\]
    \item Use the same calculation as (3).
    \item It suffices to prove that $p_A = s_a p_{r(A, a)} s_a^{\ast}$ is preserved. We can derive (5) easily from (2) by considering the inverse $r(A, a^{-1})$.
    \end{enumerate}

	By the universality of $L_R(\mathcal E, \mathcal L, \mathcal B_k)$, the map $\varphi_k: L_R(\mathcal E, \mathcal L, \mathcal B_k) \rightarrow M_{n_k}(R[z, z^{-1}])$ is well-defined. 
    
    We can see that $\varphi_k$ is surjective because \[\varphi_k(s_a^{n_k}) = z I_{n_k}\]\[\varphi_k((s_a^{\ast})^{n_k}) = z^{-1} I_{n_k}\] \[\varphi_k(p_{C(k, i)}(s_a^{\ast})^j) = E_{(i, i-j)} \text{ for all } 0 \leq j \leq i\]\[\varphi_k(p_{C(k, i)}s_a^j) = E_{(i, i+j)} \text { for all } 0 \leq j < n_k - i\]  which generate all elements of $M_{n_k}(R[z, z^{-1}])$.
    
    To show that $\varphi_k$ is injective, we use the graded uniqueness theorem \cite[Corollary~5.5]{Boava2021LeavittPA}. In particular, we grade $M_{n_k}(R[z, z^{-1}])$ with respect to $\mathbb Z$ by taking $E_{ij}z^r \mapsto n_kr + (j-i)$. We can see that the degree of $\varphi_k(p_{C_{(k, i)}})$ is $0$ and the degrees of $\varphi_k(s_a)$ and $\varphi_k(s_a^{\ast})$ are $1$ and $-1$ respectively. Thus, $\varphi_k$ is a $\mathbb Z$-graded homomorphism and is injective by the graded uniqueness theorem as $\varphi_k(rp_{C_{(k, i)}}) = rE_{(i, i)} \neq 0$ for any $r \in R \setminus \{0\}$ and $i \in [0, n_k)$ and every $B \in \mathcal B_k$ can be written as a finite disjoint union of $C_{(k, i)}$ by definition.
    
    Combining the two facts above, we find that $\varphi_k$ is a well-defined isomorphism of $R$-algebras \[L_R(\mathcal E, \mathcal L, \mathcal B_k) \cong M_{n_k}(R[z, z^{-1}]).\]

	We will now prove the direct limit. For all $k \geq 0$, there are natural maps \[i_k: L_R(\mathcal E, \mathcal L, \mathcal B_k) \hookrightarrow L_R(\mathcal E, \mathcal L, \mathcal B)\] \[\pi_k:L_R(\mathcal E, \mathcal L, \mathcal B_k) \hookrightarrow L_R(\mathcal E, \mathcal L, \mathcal B_{k+1})\] induced by the obvious maps \[p_B \mapsto p_B\] \[s_a \mapsto s_a\] \[s_a^{\ast} \mapsto s_a^{\ast}.\] These maps are injective by the graded uniqueness theorem and \cite[Lemma~4.12 (ii)]{Boava2021LeavittPA}. Furthermore, the following diagram commutes.

    \[\begin{tikzcd}
	L_R(\mathcal E, \mathcal L, \mathcal B_k) \arrow{rr}{\pi_k} \arrow[swap]{rd}{i_k}
	&  &  L_R(\mathcal E, \mathcal L, \mathcal B_{k+1}) \arrow{ld}{i_{k+1}} \\
		& L_R(\mathcal E, \mathcal L, \mathcal B)
	\end{tikzcd}\]

    Every element of $\mathcal B$ is in some $\mathcal B_k$ and all $L_R(\mathcal E, \mathcal L, \mathcal B_k)$ contain $\{s_a, s_a^{\ast}\}$. The union of the generating sets at each level $\{p_B\}_{B \in \mathcal B_k} \cup \{s_a, s_a^{\ast}\}$ is the generating set of $L_R(\mathcal E, \mathcal L, \mathcal B)$, hence \[L_R(\mathcal E, \mathcal L, \mathcal B) \cong \varinjlim L_R(\mathcal E, \mathcal L, \mathcal B_k) \cong \varinjlim M_{n_k}(R[z, z^{-1}])\] where we calculate the connecting maps as \[\beta_k = \varphi_{k+1} \circ \pi_k \circ \varphi_k^{-1}.\]
	\end{proof}

	\begin{remark}
	Note that the identity is identified with the same element for all $M_{n_k}(R[z, z^{-1}])$ and $L_R(\mathcal E, \mathcal L, \mathcal B)$. Hence, we write $1$ to refer to the identity in any of $M_{n_k}(R[z, z^{-1}])$ or $L_R(\mathcal E, \mathcal L, \mathcal B)$ and this notation is consistent.
	\end{remark}

    \begin{corollary}\cite[cf. Theorem~3]{bunce1975family} \label{notlocallymatricialtheorem}
    For any field $K$, $B_K(\{n_k\})$ is not locally matricial.
    \end{corollary}
    \begin{proof}
    It is well known that every element of a locally matricial algebra is algebraic over $K$. We will show that the element $s_a$ is transcendental over $K$ when viewing all operations as lying within $M_{n_0}(K[z, z^{-1}])$.

    If $s_a$ is not transcendental, then there exists some non-zero polynomial $\sum_{i=0}^m k_i x^i \in K[x]$ such that $\sum_{i=0}^m k_i s_a^i = 0$. Using $s_a^{n_0} = z I_{n_0}$, we take the indices $i$ modulo $n_0$ and write \[0 = \sum_{i=0}^m k_i s_a^i  = \sum_{j=0}^{n_0-1} s_a^j f_j(z) = 0\] where $f_j \in K[z]$.

    Using the same grading $z^r E_{ij} \mapsto rn_0 + (j-i)$ as before, we see that the degrees of all homogeneous components of $f_j(z) I_{n_0}$ are congruent to $0$ modulo $n_0$ and the degree of $s_a^j$ is $j$. Hence, the homogeneous components of all $\{s^j_a f_j(z)\}_{j=0}^{n_0-1}$ must all have different degrees, so we must have that $s_a^j f_j(z) = 0$ for all $j$. Because $s_a$ is invertible, this can only happen when $f_j(z) = 0$ which can only occur when all $k_i = 0$, so we are done.
    \end{proof}

    \begin{theorem} \label{stablyfinitethm}
    $B_R(\{n_k\})$ is stably finite.
    \end{theorem}
    \begin{proof}
    By Theorem~\ref{matrixalgebrarep}, $B_R(\{n_k\})$ can be written as the direct limit of rings $M_{n_k}(R[z, z^{-1}])$. 
    
    It is well known \cite[Proposition~1.3.4]{Abrams2017} that $R[z, z^{-1}]$ is the Leavitt path algebra of the single vertex self-cycle graph $R_1$.
    
    Every cycle of this graph clearly has no exit, so we apply \cite[Theorem~4.13]{steinberg2025stable} to deduce that $R[z, z^{-1}]$ is stably finite. Clearly then, $M_{n_k}(R[z, z^{-1}])$ is also stably finite. Hence, by \cite[Proposition~3.2]{steinberg2025stable}, we conclude that $B_R(\{n_k\})$ is stably finite.
    \end{proof}

    \begin{theorem} \label{notvonneumannregular}
	$B_R(\{n_k\})$ is not von Neumann regular.
	\end{theorem}
	\begin{proof}
	Let $L_R(\mathcal E, \mathcal L, \mathcal B)$ be our associated Leavitt labelled path algebra. Let $a \in \mathcal A$ be the single symbol and recall the $\mathbb Z$-grading on a Leavitt labelled path algebra with $s_a \mapsto 1$, $s_a^{\ast} \mapsto -1$, and $p_B \mapsto 0$ for all $B \in \mathcal B$.

	We will prove that the element $x = 1+s_a$ is not von Neumann regular. Note that $s_a^{\ast}s_a = s_a s_a^{\ast} = 1$ so $s_a$ is a unit.
	
	Assume for the sake of contradiction that there exists $y \in L_R(\mathcal E, \mathcal L, \mathcal B)$ such that $xyx = x \Rightarrow (xy-1)x = 0$. A grading argument easily shows that $1+s_a$ has no annihilators, so we get that $xy = 1$. 

	To show that this is impossible, let $y_m$ and $y_n$ with $m \leq n$ be the non-zero homogeneous components of $y$ of smallest and largest degree, respectively. Because $s_a$ is a unit, we have that $s_ay_n \neq 0$. Hence, $y_m$ and $s_a y_n$ are the respective smallest and largest non-zero homogeneous components of $xy$ with degrees $m$ and $n+1$ respectively. However, if $xy = 1$, then we must have that $m = n+1$, which is impossible because $m \leq n$, so we are done.
	\end{proof}

    \section{An Application to Leavitt Path Algebras}
    In this section, we will prove that, for any field $K$, $B_K(\{n_k\})$ is not Morita equivalent to any Leavitt path algebra $L_R(G)$ for any graph $G$ and commutative unital ring $R$. We will only discuss the necessary theorems for Leavitt path algebras needed to prove our primary application. For a full overview of Leavitt path algebras, see \cite{Abrams2017}.

    \begin{lemma}\label{leavittconditionsall}
    For a directed graph $G$ and unital commutative ring $R$, if $L_R(G)$ is simple and stably finite, then $L_R(G)$ is von Neumann regular.
    \end{lemma}
    \begin{proof}
    In order for $L_R(G)$ to be simple, we know that $R$ must be a field \cite[Proposition~6.2]{larki2015ideal}. Let $K = R$ be this field.

    By \cite[Theorem~2.9.1]{Abrams2017}, $L_K(G)$ is simple only if every cycle in $G$ has an exit. By \cite[Theorem~4.13]{steinberg2025stable}, $L_K(G)$ is stably finite only if every cycle of $G$ has no exit. Hence, $L_K(G)$ is simple and stably finite only if every cycle in $G$ must both have and not have an exit, which only happens when $G$ is acyclic. The fact that acyclicity of $G$ implies von Neumann regularity of $L_K(G)$ is well known \cite[Theorem~3.4.1]{Abrams2017}.
    \end{proof}

    \begin{lemma}\label{moritaequivalenceall}
    The following properties are invariant under Morita equivalence for rings with local units:
    \begin{enumerate}
    \item Simplicity
    \item Von Neumann regularity
    \item Stable finiteness
    \end{enumerate}
    \end{lemma}
    \begin{proof}
    (1) and (2) follow directly from \cite[Proposition~3.1 and Proposition~3.3]{AnhMoritaEquivalenceLocal}. 
    
    To prove (3), we will use \cite[Theorem~2.2]{molina2016morita}, which states that, for rings with local units, it suffices to prove that stable finiteness is preserved under corners and taking matrices. Let $A$ be a stably finite ring with local units. Clearly, $M_m(A)$ is stably finite because for any $n \geq 1$ we have that $M_n(M_m(A)) \cong M_{nm}(A)$ which is directly finite because $A$ is stably finite. To see that stable finiteness is preserved under corners, let $e \in A$ be an idempotent and consider $M_n(eAe) = (eI_n)(M_n(A))(eI_n)$. Because $A$ is stably finite, $M_n(A)$ is directly finite. $eI_n$ is clearly an idempotent, so by \cite[Proposition~3.1]{steinberg2025stable} we have that every idempotent in $(eI_n)(M_n(A))(eI_n)$ is finite. But that just means that $M_n(eAe)$ is directly finite, so we are done.

    \end{proof}

    \begin{theorem}
    There exists no unital commutative ring $R$, directed graph $G$, and field $K$ such that 
    \[L_R(G) \equiv_M B_K(\{n_k\}).\] Namely, the class of Leavitt labelled path algebras is strictly larger than the class of Leavitt path algebras, even up to Morita equivalence.
    \end{theorem}
    
    \begin{proof}
    By the results in Section~\ref{propertiessection}, we know that $B_K(\{n_k\})$ is simple, stably finite, and not von Neumann regular. By Lemma~\ref{leavittconditionsall}, we know that every Leavitt path algebra that is simple and stably finite must be von Neumann regular. Hence, by Lemma~\ref{moritaequivalenceall}, $B_K(\{n_k\})$ cannot be Morita equivalent to any Leavitt path algebra, so we are done.
    \end{proof}

	\section*{Declarations}
	During the preparation of this work, the author(s) used Copilot and ChatGPT for generating several proof ideas, collecting references, and proofreading the manuscript. The author(s) reviewed and edited the output as needed and take full responsibility for the content of the published article. 
	\bibliographystyle{abbrv}
	\bibliography{LabelledMoritaEquiv.bib}
	
\end{document}